\documentclass[12pt]{amsart}
\usepackage{latexsym,amsthm,amsmath,amssymb,amsfonts}
\usepackage{tikz}\usetikzlibrary{calc,matrix,arrows}

\theoremstyle{plain}\newtheorem{thm}{Theorem}[section]

\newtheorem{prop}[thm]{Proposition}

\theoremstyle{definition}\newtheorem{defn}[thm]{Definition}
\newtheorem{rem}[thm]{Remark}

\newcommand{\Z}{\mathbb{Z}}\newcommand{\Q}{\mathbb{Q}}
\newcommand{\R}{\mathbb{R}}\newcommand{\C}{\mathbb{C}}
\newcommand{\U}{\mathbb{U}}\newcommand{\D}{\mathbb{D}}
\newcommand{\QP}{\mathbb{Q}\mathrm{P}}
\newcommand{\RP}{\mathbb{R}\mathrm{P}}
\newcommand{\CP}{\mathbb{C}\mathrm{P}}
\newcommand{\ass}{\mathsf{As}}\newcommand{\cyc}{\mathsf{Cy}}
\newcommand{\aut}{\mathsf{Aut}}
\newcommand{\into}{\hookrightarrow}

\newcommand{\poly}[2]{(0:#2) \foreach \x in {1,2,...,#1} {--({\x*360/#1}:#2)}--cycle;}
\newcommand{\uarc}[2]{\draw (#1*1cm,0) arc (0:180:#2*5mm);}
\newcommand{\upt}[2]{\node[label=below:$\frac{#1}{#2}$] at (#1/#2,0) {};}
\newcommand{\darc}[2]{\draw (#1:1cm) arc ({#1+270}:{#1+#2+90}:{tan(#2/2)*1cm});}
\newcommand{\rarc}[1]{\draw (0:1cm) arc (270:{90+atan((2*#1)/(#1*#1-1))}:{1/#1*1cm});}
\newcommand{\rrarc}[2]{\begin{scope}[rotate={atan(#2)}] \rarc{#1} \end{scope}}
\newcommand{\drawDot}[1]{\filldraw[black] #1 circle (1.2pt);}
\newcommand{\drawPoly}[1]{\fill[color=blue!20] #1; \draw #1;}
\begin{document}

\title{The infinite cyclohedron and its automorphism group} 
\author{Ariadna Fossas Tenas}
\address{EPFL, Lausanne, Switzerland CH-1015}
\email{ariadna.fossastenas@epfl.ch}
\thanks{Research by the first author was supported by ERC grant
agreement number 267635 - RIGIDITY}

\author{Jon McCammond}
\address{Dept. of Math., University of California, Santa Barbara, CA 93106} 
\email{jon.mccammond@math.ucsb.edu}

\date{\today}

\begin{abstract}
  Cyclohedra are a well-known infinite familiy of finite-dimensional
  polytopes that can be constructed from centrally symmetric
  triangulations of even-sided polygons.  In this article we introduce
  an infinite-dimensional analogue and prove that the group of
  symmetries of our construction is a semidirect product of a
  degree~$2$ central extension of Thompson's infinite finitely
  presented simple group~$T$ with the cyclic group of order~$2$.
  These results are inspired by a similar recent analysis by the first
  author of the automorphism group of an infinite-dimensional
  associahedron.
\end{abstract}


\maketitle

Associahedra and cyclohedra are two well-known families of
finite-dimensional polytopes with one polytope of each type in each
dimension.  In \cite{Fo-assoc} the first author constructed an
infinite-dimensional analogue $\ass^\infty$ of the associahedra
$\ass^n$ and proved that its automorphism group is a semidirect
product of Richard Thompson's infinite finitely presented simple group
$T$ with the cyclic group of order~$2$.  In this article we extend
this analysis to an infinite-dimensional analogue $\cyc^\infty$ of the
cyclohedra $\cyc^n$.  Its symmetry group is a semidirect product of a
degree~$2$ central extension of Thompson's group~$T$ with the cyclic
group of order~$2$.  In atlas notation (which we recall in
Section~\ref{sec:thompson}), the symmetry group of the infinite
associahedron has shape $T.2$ and the symmetry group of the infinite
cyclohedron has shape $2.T.2$.

\medskip
The article is structured as follows.  The first two sections recall
basic properties of associahedra and cyclohedra, especially the way
their faces correspond to partial triangulations of convex polygons.
Section~\ref{sec:farey} reviews the Farey tessellation of the
hyperbolic plane as a convenient tool for organizing our
constructions.  The infinite-dimensional polytopes $\ass^\infty$ and
$\cyc^\infty$ are constructed in Section~\ref{sec:infinite}, and the
final sections use the close connection between Thompson's group $T$
and the Farey tessellation to analyze the symmetry group of
$\cyc^\infty$.

\section{Associahedra}\label{sec:associahedra}

In this section we recall the basic properties of associahedra.  We
begin with the construction of their face lattice.

\begin{defn}[Diagonals]
  Let $P$ be a convex polygon with $n$ vertices.  Of the
  $\binom{n}{2}$ edges connecting distinct vertices of $P$, $n$ of
  them are \emph{boundary edges} connecting vertices that occur
  consecutively in the boundary of $P$.  The other edges are called
  \emph{diagonals}.  Two diagonals are said to \emph{cross} if they
  intersect in the interior of $P$ and they are \emph{noncrossing}
  otherwise, even if they have an endpoint in common.
\end{defn}

\begin{defn}[Triangulations]
  A \emph{partial triangulation of $P$} is given by its boundary edges
  plus a possibly empty collection of pairwise noncrossing diagonals.
  A \emph{triangulation of $P$} is a maximal partial triangulation, a
  condition equivalent to the requirement that every complementary
  region is a triangle.  The number of triangulations is given by the
  Catalan numbers.  We order the partial triangulations of $P$ by
  reverse inclusion so that the triangulations are minimal elements
  and the partial triangulation with no diagonals is the unique
  maximal element.
\end{defn}

\begin{figure}
  \begin{tikzpicture}[join=round]
    \def\Ra{2cm} \def\Rb{1.618cm} \def\ra{3mm} \def\rb{1.5cm}
    \begin{scope}[xshift=-2cm]
      \draw[-,thick] (0,1)--(0,-1);
      \drawDot{(0,1)} \drawDot{(0,-1)}
      \def\p{[rotate=45] \poly{4}{\ra}}
      \begin{scope}[xshift=-5mm] \drawPoly{\p} \end{scope}
      \begin{scope}[yshift=-1.5cm] \drawPoly{\p} \draw (45:\ra)--(225:\ra); \end{scope}
      \begin{scope}[yshift=1.5cm] \drawPoly{\p} \draw (-45:\ra)--(135:\ra); \end{scope}
    \end{scope}

    \begin{scope}[xshift=2cm]
      \def\c{[rotate=90] \poly{5}{\rb}}
      \fill[color=yellow!30] \c; \draw \c;
      \foreach \x in {1,2,3,4,5} {\drawDot{(90+72*\x:\rb)}}
      \def\p{[rotate=270] \poly{5}{\ra}} \drawPoly{\p}
      \foreach \x in {0,1,2,3,4} {
        \begin{scope}[shift={+(90+72*\x:\Ra)},rotate={144*\x}]
          \drawPoly{\p} \draw (126:\ra)--(270:\ra)--(54:\ra);
        \end{scope}
        \begin{scope}[shift={+(270+72*\x:\Rb)},rotate={144*\x}]
          \drawPoly{\p} \draw (198:\ra)--(342:\ra);
        \end{scope}}
    \end{scope}
  \end{tikzpicture}
  \caption{The low dimensional associahedra $\ass^1$ and $\ass^2$.\label{fig:assoc}}
\end{figure}
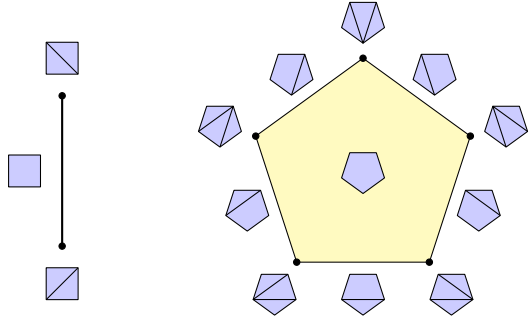

Although it is not immediately obvious, it turns out that the poset of
partial triangulations of a convex $n$-sided polygon is the face
lattice of a simple $(n-3)$-dimensional polytope called an
\emph{associahedron}.  In low-dimensions we see that a square has two
triangulations, a pentagon has five and a hexagon has fourteen.  The
corresponding $1$-, $2$-, and $3$-dimensional associahedra with $2$,
$5$ and $14$ vertices, respectively, are shown in
Figures~\ref{fig:assoc} and~\ref{fig:assoc-3}.  Although there have
been many distinct polytopal realizations of associahedra over the
years, we confine ourselves to a brief description of one particular
construction: the associahedron as the secondary polytope of a regular
convex polygon, a very general construction developed by Gelfand,
Kapranov and Zelevinsky \cite{GeKaZe89, GeKaZe90, GeKaZe94}.

\begin{figure}
  \begin{tikzpicture}[join=round,x={(.9,0)},z={(-.25,-.15)},scale=1.2]
    \path 
    (2,0,0) coordinate (1) {} 
    (1,0,{sqrt(3)}) coordinate (2) {}
    (-1,0,{sqrt(3)}) coordinate (3) {}
    (-2,0,0) coordinate (4) {}
    (-1,0,{-sqrt(3)}) coordinate (5) {}
    (1,0,{-sqrt(3)}) coordinate (6) {}
    (0,1,{sqrt(3)}) coordinate (7) {}
    (0,-1,{sqrt(3)}) coordinate (8) {}
    (-1.5,1,{-sqrt(3)/2}) coordinate (9) {}
    (-1.5,-1,{-sqrt(3)/2}) coordinate (10) {}
    (1.5,1,{-sqrt(3)/2}) coordinate (11) {}
    (1.5,-1,{-sqrt(3)/2}) coordinate (12) {}
    (0,2,0) coordinate (13) {}
    (0,-2,0) coordinate (14) {}
    ;
    \def\c{(1)--(12)--(14)--(10)--(4)--(9)--(13)--(11)--cycle};
    \fill[color=yellow!30] \c; \draw \c;
     \begin{scope}[color=blue!50,dashed]
       \draw (5)--(6);
       \draw (9)--(5)--(10);
       \draw (12)--(6)--(11);
       \filldraw[blue] (5) circle (1.2pt);
       \filldraw[blue] (6) circle (1.2pt);
     \end{scope}
     \begin{scope}[thick]
       \draw (1)--(2) (3)--(4);
       \draw (2)--(7)--(3)--(8)--(2);
       \draw (7)--(13) (8)--(14);
     \end{scope}
     \foreach \x in {1,2,3,4,7,8,...,14} {\drawDot{(\x)}}
     \def\ra{2mm} \def\p{\poly{6}{\ra}}
     \begin{scope}[shift={(2.4,0,0)}] 
       \drawPoly{\p} \draw (240:\ra)--(0:\ra)--(180:\ra)--(60:\ra);
     \end{scope}
     \begin{scope}[shift={(-2.4,0,0)}] 
       \drawPoly{\p} \draw (240:\ra)--(120:\ra)--(300:\ra)--(60:\ra);
     \end{scope}
     \begin{scope}[shift={(0,0,{sqrt(3)})}] 
       \drawPoly{\p} \draw (60:\ra)--(240:\ra);
     \end{scope}
     \begin{scope}[shift={(0,2.4,0)}] 
       \drawPoly{\p} \draw (240:\ra)--(0:\ra)--(120:\ra)--cycle;
     \end{scope}
     \begin{scope}[shift={(0,-2.4,0)}] 
       \drawPoly{\p} \draw (60:\ra)--(180:\ra)--(300:\ra)--cycle;
     \end{scope}
     \begin{scope}[shift={(.9,.8,{.3*sqrt(3)})}]
       \drawPoly{\p} \draw (0:\ra)--(240:\ra);
     \end{scope}
     \begin{scope}[shift={(.9,-.8,{.3*sqrt(3)})}] 
       \drawPoly{\p} \draw (180:\ra)--(60:\ra);
     \end{scope}
     \begin{scope}[shift={(-.9,.8,{.3*sqrt(3)})}] 
       \drawPoly{\p} \draw (120:\ra)--(240:\ra);
     \end{scope}
     \begin{scope}[shift={(-.9,-1.8,{-.3*sqrt(3)})}] 
       \drawPoly{\p} \draw (60:\ra)--(300:\ra)--(180:\ra);
     \end{scope}
     \begin{scope}[shift={(.9,-1.8,{-.3*sqrt(3)})}] 
       \drawPoly{\p} \draw (60:\ra)--(180:\ra)--(300:\ra);
     \end{scope}
     \begin{scope}[shift={(.9,1.8,{-.3*sqrt(3)})}] 
       \drawPoly{\p} \draw (120:\ra)--(0:\ra)--(240:\ra);
     \end{scope}
     \begin{scope}[shift={(1.8,1.2,{-.6*sqrt(3)})}] 
       \drawPoly{\p} \draw (120:\ra)--(0:\ra)--(240:\ra) (180:\ra)--(0:\ra);
     \end{scope}
     \begin{scope}[shift={(1.8,-1.2,{-.6*sqrt(3)})}] 
       \drawPoly{\p} \draw (60:\ra)--(180:\ra)--(300:\ra) (180:\ra)--(0:\ra);
     \end{scope}
     \begin{scope}[shift={(-1.8,1.2,{-.6*sqrt(3)})}] 
       \drawPoly{\p} \draw (0:\ra)--(120:\ra)--(240:\ra) (120:\ra)--(300:\ra);
     \end{scope}
     \begin{scope}[shift={(-1.8,-1.2,{-.6*sqrt(3)})}] 
       \drawPoly{\p} \draw (60:\ra)--(300:\ra)--(180:\ra) (120:\ra)--(300:\ra);
     \end{scope}
  \end{tikzpicture}
  \caption{The $3$-dimensional associahedron $\ass^3$ with some of its
    face labels.\label{fig:assoc-3}}
\end{figure}
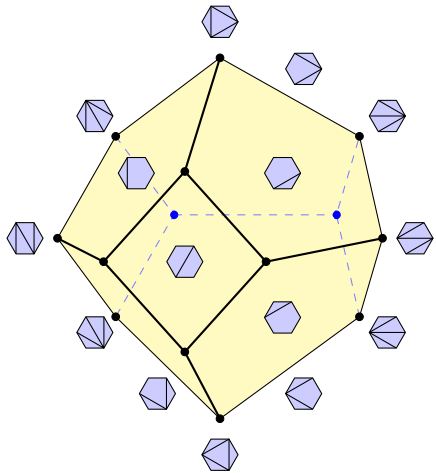

\begin{defn}[Secondary polytopes]
  The general construction of a secondary polytope goes as follows.
  Let $P$ be a $d$-dimensional convex polytope with $n+d+1$ vertices.
  For each triangulation $t$ of $P$, one defines a \emph{GKZ vector}
  $v(t) \in \R^{n+d+1}$ as the vector whose $i$-th coordinate is the
  sum of the volumes of the top-dimensional simplices in the
  triangulation $t$ that contain the $i$-th vertex of $P$.  The convex
  hull of the GKZ vectors for all possibly triangulations of $P$
  produces an $n$-dimensional polytope called its \emph{secondary
    polytope}.
\end{defn}

In their articles, Gelfand, Kapranov and Zelevinsky prove that the
secondary polytope of a convex polygon is an associahedron.  The
reader should note that the indices shift during the construction: the
regular polygon with $n$ vertices produces the associahedron of
dimension $n-3$.  Secondary polytopes were later generalized by
Billera and Sturmfels to produce fiber polytopes \cite{BiSt92} and
both of these constructions are described in Ziegler's book on
polytopes \cite{Zi95}.  For a survey of other polytopal realizations
of associahedra and for further references to the literature, see the
excellent recent article by Ceballos, Santos and Ziegler
\cite{CeSaZi13}.  In this article, we shall insist that our
associahedra are constructed as metric secondary polytopes from
regular polygons of a specific size.

\begin{defn}[Associahedra]\label{def:assoc}
  A \emph{standard regular polygon} is one obtained as the convex hull
  of equally spaced points on the unit circle (i.e. a regular polygon
  circumscribed by the unit circle) and a \emph{standard
    associahedron} is the metric polytopal realization of the
  associahedron obtained by applying the secondary polytope
  construction to a standard regular polygon.  We write $\ass^n$ to
  denote the standard associahedron of dimension~$n$, following the
  convention that uses superscripts to indicate the dimension of a
  polytope.
\end{defn}

By standardizing metrics in this way, the standard associahedra are
symmetric polytopes and there exist some isometric embeddings between
them.

\begin{prop}[Isometries and embeddings]\label{prop:ass-isom-embed}
  Every isometry of the standard regular polygon $P_n$ extends to an
  isometry of the standard associahedron $\ass^{n-3}$, which means
  that the dihedral group of size $2n$ acts isometrically on
  $\ass^{n-3}$.  In addition, the natural inclusion map $P_n \into
  P_{2n}$ which sends the vertices of $P_n$ to every other vertex of
  $P_{2n}$ induces an isometric inclusion of $\ass^{n-3}$ into
  $\ass^{2n-3}$ as one of its faces.
\end{prop}

\begin{proof}
  The first assertion is immediate from the definition of the
  construction.  For the second assertion we identify the (partial)
  triangulations of $P_n$ with (partial) triangulations of $P_{2n}$
  that include the diagonals that are images of the boundary edges of
  $P_n$.  In the coordinates of the construction, this involves
  including $\R^n$ into $\R^{2n}$ and translating by a vector which
  records the contributions of the $n$ triangles that are outside the
  image of $P_n$.  Since the same $n$ exterior triangles always occur,
  this is a pure translation and the embedding is an isometry.
\end{proof}

\begin{rem}[Other isometric embeddings]
  There exists an isometric embedding of $\ass^{n-3}$ into
  $\ass^{m-3}$ whenever $n$ divides $m$, but when $\frac{m}{n}$ is
  greater than $2$, one needs to first choose a consistent
  triangulation of the portion of $P_m$ exterior to the image of $P_n$
  in order to define the map.  In other words, there are multiple
  isometric embeddings of $\ass^{n-3}$ into $\ass^{m-3}$ in this case
  as opposed to the canonical embedding described in the proposition.
\end{rem}

\section{Cyclohedra}\label{sec:cyclohedra}

In this section we shift our attention from associahedra to
cyclohedra.  The cyclohedra were initially investigated by Raul Bott
and Clifford Taubes in \cite{BoTa94} in a slightly different guise
before being rediscovered by Rodica Simion in \cite{Si03}.  In the
same way that the faces of an associahedron correspond to (partial)
triangulations of a polygon, the faces of a cyclohedron correspond to
centrally symmetric (partial) triangulations of an even-sided polygon.

\begin{defn}[Centrally symmetric triangulations]
  Let $P=P_{2n}$ be a standard $2n$-sided polygon and consider the
  $\pi$-rotation of $P$ that sends the vertex $v_i$ to the vertex
  $v_{i+n}$ (where vertices are numbered modulo $2n$ in the order they
  occur in the boundary of $P$).  As an isometry it sends every vector
  to its negative, i.e. it sends $x+iy$ to $-x-iy$.  A partial
  triangulation of $P$ is said to be \emph{centrally symmetric} when
  it is invariant under this $\pi$-rotation.  The centrally symmetric
  partial triangulations form a poset as before under reverse
  inclusion on the sets of diagonals used to define them.  The minimal
  elements of this poset are the \emph{centrally symmetric
    triangulations}.
\end{defn}

The poset of centrally symmetric partial triangulations of $P_{2n}$ is
the face lattice of a simple $(n-1)$-dimensional polytope called a
\emph{cyclohedron}.

\begin{defn}[Cyclohedra]
  As should be clear from the definitions, the vertices of the
  cyclohedron can be identified with a subset of the vertices of the
  associahedron built from the same regular polygon.  More precisely,
  the triangulations of a standard polygon with $2n$ sides index the
  vertices of the associahedron $\ass^{2n-3}$ and the subset of
  centrally symmetric ones index the vertices of cyclohedron whose
  dimension is merely $n-1$.  An additional reason for using the
  secondary polytope construction to build the associahedron
  $\ass^{2n-3}$ is that the convex hull of its vertices with centrally
  symmetric labels is a polytopal realization of the
  $(n-1)$-dimensional cyclohedron.  We call this the \emph{standard
    cyclohedron} in dimension $n-1$ and we write $\cyc^{n-1}$ for this
  polytope.
\end{defn}

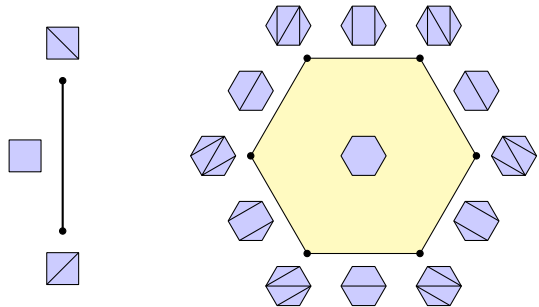
\begin{figure}
  \begin{tikzpicture}[join=round]
    \def\Ra{2cm} \def\Rb{1.732cm} \def\ra{3mm} \def\rb{1.5cm}
    \begin{scope}[xshift=-2cm]
      \node (a) at (0,-1) {};
      \node (b) at (0,1) {};
      \draw[-,thick] (0,1)--(0,-1);
      \foreach \x in {a,b} {\drawDot{(\x)}}
      \def\p{\poly{4}{\ra}}
      \begin{scope}[yshift=-1.5cm,rotate=45]
        \fill[color=blue!20] \p; \draw \p;
        \draw (0:\ra)--(180:\ra);
      \end{scope}
      \begin{scope}[yshift=1.5cm,rotate=45]
        \fill[color=blue!20] \p; \draw \p;
        \draw (90:\ra)--(270:\ra);
      \end{scope}
      \begin{scope}[xshift=-5mm,rotate=45]
        \fill[color=blue!20] \p; \draw \p;
      \end{scope}
    \end{scope}

    \begin{scope}[xshift=2cm]
      \def\c{\poly{6}{\rb}}
      \def\p{\poly{6}{\ra}}
      \fill[color=yellow!30] \c; \draw \c;
      \foreach \x in {1,2,...,6} {\drawDot{(60*\x:1.5cm)}}
      \fill[color=blue!20] \p; \draw \p;
      \foreach \x in {0,120,240} {\begin{scope}[rotate=\x]
          \begin{scope}[shift={+(0:\Ra)}]
            \fill[color=blue!20] \p; \draw \p;
            \draw (0:\ra)--(120:\ra)--(300:\ra)--(180:\ra);
          \end{scope}
          \begin{scope}[shift={+(30:\Rb)}]
            \fill[color=blue!20] \p; \draw \p;
            \draw (120:\ra)--(300:\ra);
          \end{scope}
          \begin{scope}[shift={+(60:\Ra)}]
            \fill[color=blue!20] \p; \draw \p;
            \draw (240:\ra)--(120:\ra)--(300:\ra)--(60:\ra);
          \end{scope}
          \begin{scope}[shift={+(90:\Rb)}]
            \fill[color=blue!20] \p; \draw \p;
            \draw (300:\ra)--(60:\ra) (240:\ra)--(120:\ra);
          \end{scope}
      \end{scope}}
    \end{scope}
  \end{tikzpicture}
  \caption{The low-dimensional cyclohedra $\cyc^1$ and $\cyc^2$.\label{fig:cyclo}}
\end{figure}

We should note that for many of the polytopal realizations of the
associahedron, the convex hull of the vertices with centrally
symmetric labels is not a cyclohedron.  The convex hull has a
different combinatorial type and often a different and higher
dimension \cite{HoLa07}.  Using the secondary polytope metric, the
cyclohedron $\cyc^{n-1}$ inside the associahedron $\ass^{2n-3}$ can
also be described in terms of its isometries.

\begin{prop}[Fixed subpolytope]\label{prop:fixed}
  Let $P$ be a standard $2n$-sided polygon and let $\ass^{2n-3}$ be
  the standard associahedron constructed from the triangulations of
  $P$.  The $\pi$-rotation of $P$ induces an isometric involution
  $\tau$ that acts on $\R^{2n}$ by permuting coordinates, it stablizes
  $\ass^{2n-3} \subset \R^{2n}$ and the intersection of the fixed set
  of $\tau$ with $\ass^{2n-3}$ is the cyclohedron $\cyc^{n-1}$.
\end{prop}

\begin{figure}
  \begin{tikzpicture}[join=round,x={(.9,0)},z={(-.25,-.15)},scale=1.2]
    \path 
    (2,0,0) coordinate (1) {} 
    (1,0,{sqrt(3)}) coordinate (2) {}
    (-1,0,{sqrt(3)}) coordinate (3) {}
    (-2,0,0) coordinate (4) {}
    (-1,0,{-sqrt(3)}) coordinate (5) {}
    (1,0,{-sqrt(3)}) coordinate (6) {}
    (0,1,{sqrt(3)}) coordinate (7) {}
    (0,-1,{sqrt(3)}) coordinate (8) {}
    (-1.5,1,{-sqrt(3)/2}) coordinate (9) {}
    (-1.5,-1,{-sqrt(3)/2}) coordinate (10) {}
    (1.5,1,{-sqrt(3)/2}) coordinate (11) {}
    (1.5,-1,{-sqrt(3)/2}) coordinate (12) {}
    (0,2,0) coordinate (13) {}
    (0,-2,0) coordinate (14) {}
    ;
    \def\c{(1)--(12)--(14)--(10)--(4)--(9)--(13)--(11)--cycle};
    \fill[color=yellow!30] \c; \draw \c;
     \begin{scope}[color=blue!50,dashed]
       \draw (5)--(6);
       \draw (9)--(5)--(10);
       \draw (12)--(6)--(11);
       \filldraw[blue] (5) circle (1.2pt);
       \filldraw[blue] (6) circle (1.2pt);
     \end{scope}
     \begin{scope}[thick]
       \draw (1)--(2) (3)--(4);
       \draw (2)--(7)--(3)--(8)--(2);
       \draw (7)--(13) (8)--(14);
     \end{scope}
     \draw[very thick, color=violet] (1)--(2)--(3)--(4);
     \draw[dashed,very thick, color=violet] (4)--(5)--(6)--(1);
     \foreach \x in {1,2,3,4,7,8,...,14} {\drawDot{(\x)}}
     \def\ra{2mm} \def\p{\poly{6}{\ra}}
     \begin{scope}[shift={(2.4,0,0)}] 
       \drawPoly{\p} \draw (240:\ra)--(0:\ra)--(180:\ra)--(60:\ra);
     \end{scope}
     \begin{scope}[shift={(-2.4,0,0)}] 
       \drawPoly{\p} \draw (240:\ra)--(120:\ra)--(300:\ra)--(60:\ra);
     \end{scope}
     \begin{scope}[shift={(0,0,{sqrt(3)})}] 
       \drawPoly{\p} \draw (60:\ra)--(240:\ra);
     \end{scope}
     \begin{scope}[shift={(0,2.4,0)}] 
       \drawPoly{\p} \draw (240:\ra)--(0:\ra)--(120:\ra)--cycle;
     \end{scope}
     \begin{scope}[shift={(0,-2.4,0)}] 
       \drawPoly{\p} \draw (60:\ra)--(180:\ra)--(300:\ra)--cycle;
     \end{scope}
     \begin{scope}[shift={(.9,.8,{.3*sqrt(3)})}]
       \drawPoly{\p} \draw (0:\ra)--(240:\ra);
     \end{scope}
     \begin{scope}[shift={(.9,-.8,{.3*sqrt(3)})}] 
       \drawPoly{\p} \draw (180:\ra)--(60:\ra);
     \end{scope}
     \begin{scope}[shift={(-.9,.8,{.3*sqrt(3)})}] 
       \drawPoly{\p} \draw (120:\ra)--(240:\ra);
     \end{scope}
     \begin{scope}[shift={(-.9,-1.8,{-.3*sqrt(3)})}] 
       \drawPoly{\p} \draw (60:\ra)--(300:\ra)--(180:\ra);
     \end{scope}
     \begin{scope}[shift={(.9,-1.8,{-.3*sqrt(3)})}] 
       \drawPoly{\p} \draw (60:\ra)--(180:\ra)--(300:\ra);
     \end{scope}
     \begin{scope}[shift={(.9,1.8,{-.3*sqrt(3)})}] 
       \drawPoly{\p} \draw (120:\ra)--(0:\ra)--(240:\ra);
     \end{scope}
     \begin{scope}[shift={(1.8,1.2,{-.6*sqrt(3)})}] 
       \drawPoly{\p} \draw (120:\ra)--(0:\ra)--(240:\ra) (180:\ra)--(0:\ra);
     \end{scope}
     \begin{scope}[shift={(1.8,-1.2,{-.6*sqrt(3)})}] 
       \drawPoly{\p} \draw (60:\ra)--(180:\ra)--(300:\ra) (180:\ra)--(0:\ra);
     \end{scope}
     \begin{scope}[shift={(-1.8,1.2,{-.6*sqrt(3)})}] 
       \drawPoly{\p} \draw (0:\ra)--(120:\ra)--(240:\ra) (120:\ra)--(300:\ra);
     \end{scope}
     \begin{scope}[shift={(-1.8,-1.2,{-.6*sqrt(3)})}] 
       \drawPoly{\p} \draw (60:\ra)--(300:\ra)--(180:\ra) (120:\ra)--(300:\ra);
     \end{scope}
  \end{tikzpicture}
  \caption{The involution $\tau$ acts on the $3$-dimensional
    associahedron $\ass^3$ as a vertical reflection through the
    equator.  The $2$-dimensional hexagonal cyclohedron $\cyc^2$ is
    visible as the intersection of $\ass^3$ with the fixed set
    of~$\tau$.\label{fig:cy2-in-ass3}}
\end{figure}
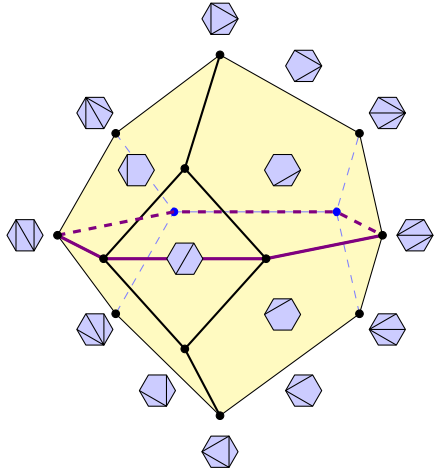

The sharp drop in dimension from the associahedron to the cyclohedron
is caused by the geometry of the involution $\tau$.

\begin{rem}[Dimension drop]\label{rem:dim-drop}
  In the secondary polytope coordinate system, the involution $\tau$
  systematically switches coordinates $x_i$ and $x_{i+n}$.  Thus,
  geometrically, it fixes an $n$-dimensional subspace of $\R^{2n}$ and
  it acts as the antipodal map on its $n$-dimensional orthogonal
  complement.  The associahedron $\ass^{2n-3}$ lives in an affine
  subspace of $\R^{2n}$ of dimension $2n-3$ and $\tau$ fixes its
  center.  If we translate the coordinate system so that the center of
  $\ass^{2n-3}$ is the new origin and restrict $\tau$ to the subspace
  containing $\ass^{2n-3}$, then $\tau$ fixes a subspace of dimension
  $n-1$ and acts as the antipodal map on its orthogonal complement of
  dimension $n-2$.  Hence the dimension of the corresponding
  cyclohedron.
\end{rem}

In low-dimensions we see that both triangulations of a square are
centrally symmetric so that $\cyc^1$ and $\ass^1$ are identical.  A
hexagon has exactly six centrally symmetric triangulations plus six
centrally symmetric partial triangulations, so $\cyc^2$ is a hexagon.
See Figure~\ref{fig:cyclo}.  We can also see $\cyc^2$ as a polytope
contained in the associahedron $\ass^3$.  See
Figure~\ref{fig:cy2-in-ass3}.  The top and bottom vertices of $\ass^3$
are swapped by $\tau$ and the ones fixed by $\tau$ are the six
vertices that appear to be on the equator relative to the north/south
pole through the top and bottom vertices.  The $3$-dimensional
cyclohedron $\cyc^3$ is embedded in the $5$-dimensional associahedron
$\ass^5$.  Its vertices are labelled by the twenty centrally symmetric
triangulations of a regular octogon, it has thirty edges and twelve
$2$-dimensional faces of which four are squares, four are pentagons
and four are hexagons.  

As an immediate consequence of Proposition~\ref{prop:fixed} we have
the following, whose proof is identical to that of
Proposition~\ref{prop:ass-isom-embed}.

\begin{prop}[Isometries and embeddings]\label{prop:cyc-isom-embed}
  Every isometry of the standard regular polygon $P_{2n}$ extends to an
  isometry of the standard cyclohedron $\cyc^{n-1}$, which means
  that the dihedral group of size $4n$ acts isometrically on
  $\cyc^{n-1}$.  In addition, the natural inclusion map $P_{2n} \into
  P_{4n}$ which sends the vertices of $P_{2n}$ to every other vertex of
  $P_{4n}$ induces an isometric inclusion of $\cyc^{n-1}$ into
  $\cyc^{2n-1}$ as one of its faces.
\end{prop}

\section{Farey tessellations}\label{sec:farey}

In this section we describe the well-known Farey tessellation of the
hyperbolic plane and some slight variations.  We begin with the
traditional version.

\begin{defn}[Rational Farey tessellation]
  As is well-known, the group $PSL_2(\Z)$ naturally acts on the circle
  $\RP^1$ sitting inside the $2$-sphere $\CP^1$ and it acts
  transitively on points in $\QP^1$.  In addition, the action
  preserves the orientation of the circle $\RP^1$ and thus stabilizes
  each hemisphere of $\CP^1$.  In other words, $PSL_2(\Z)$ acts on the
  upper-half plane model of the hyperbolic plane.  The orbit of the
  bi-infinite geodesic connecting the points $0 = \frac{0}{1}$ and
  $\infty = \frac{1}{0}$ in the boundary under this action divides the
  plane into geodesics and ideal triangles.  This is what we call the
  \emph{rational Farey tessellation of the upper-half plane $\U$} and
  a portion of it is shown in Figure~\ref{fig:rat-upper}.
\end{defn}

\begin{figure}
  \begin{tikzpicture}[scale=3]
    \fill[color=blue!20] (0,0)--(4cm,0)--(4cm,1cm)--(0,1cm)--cycle;
    \foreach \x in {1,2,3,4} {\uarc{\x}{1}}
    \foreach \x in {1,2,3,4,5,6,7,8} {\uarc{.5*\x}{1/2}}
    \foreach \x in {1,2,3,4} {\uarc{\x}{1/3} \uarc{\x}{1/4} \uarc{\x}{1/5}}
    \foreach \x in {1,3,5,7} {\uarc{\x/2}{1/6} \uarc{\x/2}{1/10}}
    \foreach \x in {1,4,7,10} {\uarc{\x/3}{1/3} \uarc{\x/3}{1/12}}
    \foreach \x in {2,5,8,11} {\uarc{\x/3}{1/6} \uarc{\x/3}{1/15}}
    \foreach \x in {1,5,9,13} {\uarc{\x/4}{1/4} \uarc{\x/4}{1/20}}
    \foreach \x in {3,7,11,15} {\uarc{\x/4}{1/12}}
    \foreach \x in {1,6,11,16} {\uarc{\x/5}{1/5}}
    \foreach \x in {2,7,12,17} {\uarc{\x/5}{1/15}}
    \foreach \x in {3,8,13,18} {\uarc{\x/5}{1/10}}
    \foreach \x in {4,9,14,19} {\uarc{\x/5}{1/20}}
    \foreach \x in {0,1,2,3,4} {\draw[-] (\x,0) -- (\x,1cm);}
    \draw[-] (0,0) -- (4,0);
    \foreach \x in {0,1,2,3,4} {\upt{\x}{1}}
    \foreach \x in {1,3,5,7} {\upt{\x}{2}}
    \foreach \x in {1,2,4,5,7,8,10,11} {\upt{\x}{3}}
  \end{tikzpicture}
\caption{A portion of the rational Farey tessellation of the
  upper-half plane $\U$. In the interval between $0$ and $4$, the arcs
  where both endpoints have denominator at most $5$ are drawn and the
  points with denominator at most $3$ are
  labeled.\label{fig:rat-upper}}
\end{figure}
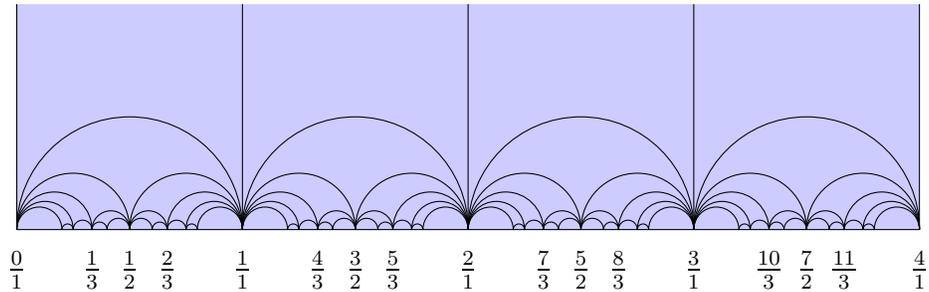

The dyadic Farey tessellation of the upper-half plane $\U$ is
combinatorially equivalent to the rational Farey tessellation, but
metrically distinct.

\begin{defn}[Dyadic Farey tessellation]
  The \emph{dyadic Farey tessellation of the upper-half plane $\U$} is
  formed by first connecting each integer $n$ in the boundary to
  $\infty = \frac{1}{0}$ and to the integer $n+1$.  Afterwards,
  intervals in the boundary are iteratively evenly subdivided and
  additional ideal triangles are drawn.  See
  Figure~\ref{fig:dyadic-upper}.  Instead of having all rational
  numbers as endpoints, the modified tessellation has endpoints at the
  dyadic rationals $\Z[\frac{1}{2}]$ plus $\infty$.
\end{defn}

Although this version looks more symmetric than the rational Farey
tessllation, it is the rational version which is invariant under a
vertex transitive group of hyperbolic isometries. 

\begin{rem}[Literature]\label{rem:literature}
  Both versions have been used to study topics related to the focus of
  this article.  The dyadic version of the Farey tessellation is
  closely related to the Belk and Brown forest diagrams for elements
  of Thompson's group $F$ \cite{BeBr05} and the rational version is at
  the heart of the cluster algebra structure on the hyperbolic plane
  that Sergey Fomin, Michael Shapiro, and Dylan Thurston produced by a
  series of edge-flips \cite{FoShTh08}.  An edge-flip is a
  modification of a triangulation which removes an edge and then adds
  in the other diagonal of the resulting quadrilateral.  Finite
  sequences of edge-flips lead to what one might call a \emph{finite
    retriangulation}.  It is also worth noting along these lines that
  the theory of cluster algebras of finite type is a major source of
  polytopal realizations of both associahedra and cyclohedra, although
  they are distinct from the realizations that we are using.
\end{rem}

Of more immediate interest are the analogous tessellations of the disc
model.

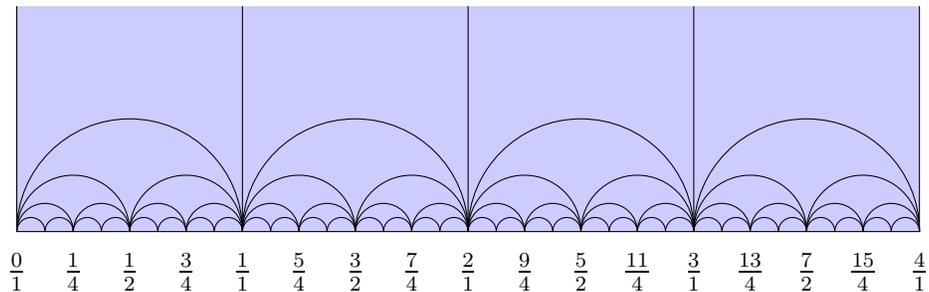
\begin{figure}
  \begin{tikzpicture}[scale=3]
    \fill[color=blue!20] (0,0)--(4cm,0)--(4cm,1cm)--(0,1cm)--cycle;
    \foreach \x in {1,2,3,4} {\uarc{\x}{1}}
    \foreach \x in {1,2,3,4,5,6,7,8} {\uarc{.5*\x}{1/2}}
    \foreach \x in {1,2,3,...,16} {\uarc{.25*\x}{1/4}}
    \foreach \x in {1,2,3,...,32} {\uarc{.125*\x}{1/8}}
    \foreach \x in {0,1,2,3,4} {\draw[-] (\x,0) -- (\x,1cm);}
    \draw[-] (0,0) -- (4,0);
    \foreach \x in {0,1,2,3,4} {\upt{\x}{1}}
    \foreach \x in {1,3,5,7} {\upt{\x}{2}}
    \foreach \x in {1,3,5,7,9,11,13,15} {\upt{\x}{4}}
  \end{tikzpicture}
\caption{A portion of the dyadic Farey tessellation of the upper-half
  plane $\U$.  In the interval between $0$ and $4$, the arcs where
  both endpoints have denominator at most $8$ are drawn and the points
  with denominator at most $4$ are labeled.\label{fig:dyadic-upper}}
\end{figure}

\begin{defn}[Tessellations of the disc]
  The \emph{rational Farey tessellation of the unit disc $\D$} is the
  image of the rational Farey tessellation of the upper-half plane
  model when translated to the Poincar\'e disc model by the linear
  fractional transformation that sends $0$, $1$ and $\infty$ to $1$,
  $i$ and $-1$, respectively.  Under this transformation, the point
  $\frac{1}{3}$ is sent to the point $\frac{3}{5} + \frac{4}{5}i$,
  which lies on the unit circle precisely because $3^2 + 4^2 = 5^2$,
  and more generally this transformation establishes a bijection
  between the points in $\QP^1$ and the rational points on the unit
  circle viewed as rescaled Pythagorean triples.  The \emph{dyadic
    Farey tessellation of the unit disc $\D$} is a combinatorially
  equivalent but metrically distinct tessellation where the
  modifications are similar to the ones described above.  After
  placing the inital geodesic connecting $1$ and $-1$, additional
  boundary points are added by evenly dividing boundary arcs and
  adding in a new ideal triangle.  We call the arcs that occur in the
  dyadic Farey tessellation, \emph{dyadic Farey arcs}.  Both the
  rational and the dyadic Farey tessellations of the unit disc $\D$
  are shown in Figure~\ref{fig:farey-disc}.  As with the upper-half
  plane tessellations, the dyadic version appears more symmetric, but
  it is the rational version which is invariant under a vertex
  transitive group of hyperbolic isometries.  
\end{defn}

\begin{figure}
  \begin{tikzpicture}[scale=2]
    \begin{scope}[xshift=-1.25cm]
    \fill[color=blue!20] (0,0) circle (1cm);
    \draw (0,0) circle (1cm);
    \draw[-] (-1,0) -- (1,0);
    \foreach \x in {-2,-1,0,1} {\darc{90*\x}{90}}
    \foreach \xs in {-1,1} {
      \foreach \ys in {-1,1} {
        \begin{scope}[xscale=\xs,yscale=\ys]
          \rarc{2} \rarc{3} 
          \rarc{4} \rarc{5}
          \rrarc{5}{12/5} 
          \rrarc{3}{4/3}
          \rrarc{8}{4/3}
          \rrarc{7}{3/4}
          \rrarc{13}{8/15}
          \rrarc{21}{5/12}
        \end{scope}
      }
    }
    \end{scope}
    \begin{scope}[xshift=1.25cm]
      \fill[color=blue!20] (0,0) circle (1cm);
      \draw (0,0) circle (1cm);
      \draw[-] (-1,0) -- (1,0);
      \foreach \x in {-2,-1,0,1} {\darc{90*\x}{90}}
      \foreach \x in {-5,-4,...,2} {\darc{45*\x}{45}}
      \foreach \x in {-11,-10,...,4} {\darc{22.5*\x}{22.5}}
      \foreach \x in {-23,-22,...,8} {\darc{11.25*\x}{11.25}}
    \end{scope}
  \end{tikzpicture}
\caption{A portion of the rational Farey tessellation of the unit disc
  $\D$ is shown on the left and a portion of the dyadic Farey
  tessellation of the unit disc $\D$ is shown on the
  right.\label{fig:farey-disc}}
\end{figure}
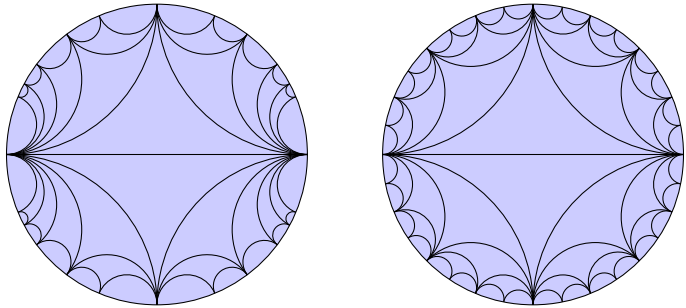

\begin{rem}[Tessellations and triangulations]\label{rem:tess-tri}
  The Farey tessellation, in any of its various forms, can also be
  viewed as a $2$-dimensional simplicial complex.  One adds a vertex
  for each endpoint of a bi-infinite geodesic (indexed by $\Q P^1 = \Q
  \cup \{\infty\}$ in the rational case and by $\Z[\frac12] \cup
  \{\infty\}$ in the dyadic case), an edge for every bi-infinite
  geodesic and a triangle for each complementary region.  Note that
  each vertex link in this $2$-complex looks like the standard
  simplicial structure on $\R$ with one vertex for each integer and
  edges connecting adjacent integers.  The group of combinatorial
  automorphisms of the reals with this cell structure is the infinite
  dihedral group and the combinatorial automorphisms of the full
  $2$-complex associated to the Farey tessellation is the non-oriented
  version of $PSL_2(\Z) \rtimes \Z_2$ which can be identified with
  $PGL_2(\Z)$.  One way to see this identification is as follows.  The
  natural action of $PGL_2(\Z)$ on $\C P^1$ is orientation-preserving.
  The images of the positive imaginary axis under this action form two
  copies of the Farey tessellation, one in the upper-half plane and
  another in the lower-half plane.  If we identify numbers with their
  complex conjugates, then $PGL_2(\Z)$ acts on the upper-half plane
  and those elements represented by matrices with a negative
  determinant become the orientation-reversing isometries of the Farey
  tessellation.
\end{rem}

And finally, notice that when the dyadic Farey tessellation is
depicted in the Klein disc model of the hyperbolic plane, it looks
like a nested union of standard $n$-sided regular polygons where $n =
2^k$ for all $k>1$.  It is this close connection between standard
regular polygons and the dyadic Farey tessellation of the unit disc
which enables us to construct an infinite-dimensional version of a
cyclohedron and to compute the structure of its automorphism group.

\section{Infinite-dimensional polytopes}\label{sec:infinite}

In this section we define what we mean by an infinite-dimensional
polytope and illustrate the definition by constructing the polytopes
that we call $\ass^\infty$ and $\cyc^\infty$.

\begin{defn}[Infinite-dimensional polytopes]
  Let $n_1 < n_2 < \cdots$ be an increasing list of positive integers,
  let $P^{n_i}$ be an $n_i$-dimensional metric polytope and for each
  $i$ select a specific inclusion map $P^{n_i} \into P^{n_{i+1}}$
  which isometrically embeds $P^{n_i}$ as a face of $P^{n_{i+1}}$.
  Finally, let $P^\infty$ denote the union of such a nested
  sequence. We call $P^\infty$ an \emph{infinite-dimensional
    polytope}.
\end{defn}

We believe the name is reasonable because these infinite-dimensional
polytopes share many properties with finite-dimensional polytopes.
For example, infinite-dimensional polytopes as defined here are the
convex hull of a countable set of points in a countable direct sum of
copies of $\R$ with only finitely many points in any
finite-dimensional affine subspace.  Also, the union $P^\infty$ has a
well-defined set of faces and its face lattice is the union of the
nested face lattices for the $P^{n_i}$.  Moreover, in the same way
that every finite polytope can be viewed as a contractible regular
cell complex with one cell for each element of the face lattice,
$P^\infty$ is a contractible regular cell complex with one cell for
each element in its face lattice.  Contractibility of the union is
just a special case of the standard fact that a cell complex
constructed as a nested union of contractible cell complexes is itself
contractible. (Recall the standard proof: any map of a sphere into the
union has a compact image which is contained in a finite subcomplex
which is contained in one of the finite contractible stages in the
union, and thus this map is homotopic to a constant map.  This means
that all of the homotopy groups of the union are trivial and by
Whitehead's theorem the union is contractible.)

\begin{figure}
  \begin{tikzpicture}[node distance=15mm]
    \node(a1)  {$\ass^1$}; 
    \node(a5)  [right of=a1]  {$\ass^5$}; 
    \node(a13) [right of=a5]  {$\ass^{13}$}; 
    \node(ad1) [right of=a13] {$\cdots$}; 
    \node(an)  [right of=ad1] {$\ass^{2n-3}$}; 
    \node(ad2) [right of=an]  {$\cdots$}; 
    \node(ai)  [right of=ad2] {$\ass^\infty$}; 
    \node(c1)  [below of=a1]  {$\cyc^1$}; 
    \node(c3)  [right of=c1]  {$\cyc^3$};
    \node(c7)  [right of=c3]  {$\cyc^7$};
    \node(cd1) [right of=c7]  {$\cdots$};
    \node(cn)  [right of=cd1]  {$\cyc^{n-1}$};
    \node(cd2) [right of=cn]  {$\cdots$};
    \node(ci)  [right of=cd2]  {$\cyc^\infty$};
    \node (ae) at ($(ad2)!.5!(ai)$) {$=$};
    \node (ce) at ($(cd2)!.5!(ci)$) {$=$};
    \begin{scope}[right hook->]
      \draw (a1)--(a5); \draw (a5)--(a13); \draw (a13)--(ad1); 
      \draw (ad1)--(an); \draw (an)--(ad2); 
      \draw (c1)--(c3); \draw (c3)--(c7); \draw (c7)--(cd1); 
      \draw (cd1)--(cn); \draw (cn)--(cd2); 
      \draw (c3)--(a5); \draw (c7)--(a13); 
      \draw (cn)--(an); \draw (ci)--(ai); 
    \end{scope}
    \draw[double] (c1)--(a1); 
  \end{tikzpicture}
  \caption{The nested embeddings of finite-dimensional polytopes used
    to construct $\ass^\infty$ and $\cyc^\infty$.  The integer $n$ in
    the diagram ranges over the proper powers of
    $2$.\label{fig:infinite}}
\end{figure}

\begin{defn}[$\ass^\infty$ and $\cyc^\infty$]
  For each integer $n=2^k$ with $k>1$ we consider the standard regular
  $n$-sided polygon and the corresponding associahedra $\ass^{2n-3}$
  and cyclohedra $\cyc^{n-1}$.  As a consequence of
  Propositions~\ref{prop:ass-isom-embed}, \ref{prop:fixed}, and
  ~\ref{prop:cyc-isom-embed} we have the isometric embeddings shown in
  Figure~\ref{fig:infinite} and the unions of these nested embeddings
  define infinite dimensional polytopes that we call the
  \emph{infinite associahedron} $\ass^\infty$ and the \emph{infinite
    cyclohedron} $\cyc^\infty$.
\end{defn}

From this construction as a union of polytopes, it should be clear
that the vertices of $\ass^\infty$ are indexed by finite
retriangulations of the unit disc (in the sense described in
Remark~\ref{rem:literature}) where all but finitely many of the arcs
of the triangulation are dyadic Farey arcs, and the vertices of
$\cyc^\infty$ are those labeled by such finite retriangulations that
are invariant under the $\pi$-rotation of the disc.  In fact, as in
the finite-dimensional case, the $\pi$-rotation of the disc induces an
involution of $\ass^\infty$ and the portion of $\ass^\infty$ fixed by
this involution is~$\cyc^\infty$.

\section{Thompson's group $T$}\label{sec:thompson}

In this section we recall the definition of Richard Thompson's
infinite finitely presented simple group $T$ and its close connection
with the Farey tessellation of the hyperbolic plane.  In 1965 Richard
Thompson defined three groups conventionally denoted $F$, $T$ and $V$.
The group $F$ is the easiest to define and $T$ and $V$ were the first
known examples of infinite finitely-presented simple groups
\cite{CaFlPa76}.  Thompson's group $F$ is the group of
piecewise-linear homeomorphisms of the unit interval where every slope
is a power of $2$ and all of the break points occur at dyadic
rationals.  This group is also often studied in terms of pairs of
rooted finite planar binary trees.  Thompson's group $T$ is an
extension of $F$ where the endpoints of the interval are identified
and certain rotations are allowed.  Since we do not need a detailed
description for our main result, we shall merely sketch the key ideas.

\begin{defn}[Thompson's group $T$]\label{def:T}
  For our purposes, the main result we need about Thompson's group $T$
  is that its elements can be identified with pairs of ideal Farey
  $m$-gons in the dyadic Farey tessellation of the unit disc with an
  indication of how the corners of the one are sent to the corners of
  the other.  This description induces piecewise-linear maps on the
  boundary circle (parametrized by arc length and rescaled dividing
  by $2\pi$) \cite{FuKaSe12,Fo-diss}.  Alternatively, its elements can
  be identified with the group of finite retriangulations and
  renormalizations of the Farey tessellation of the disc.  See Sonja
  Mitchell Gallagher's dissertation for precise definitions and
  further details of this description \cite{Ga13}.
\end{defn}

There are several relevant groups that are closely related to $T$ and
for these it is convenient to introduce the following notation.

\begin{defn}[Atlas notation]\label{def:atlas}
  The atlas of finite groups popularized many simple notational
  conventions for describing groups that are closely related to simple
  groups \cite{CCNPW85}.  For example, when $G$ is a group that has a
  normal subgroup isomorphic to a group $A$ and the quotient group
  $G/A$ is isomorphic to a group $B$, one writes $G = A.B$ and says
  that $G$ has \emph{shape} $A.B$.  When iterated, the convention is
  to left associate.  Thus $G = A.B.C$ means that $G = (A.B).C$,
  i.e. $G$ has a normal subgroup of shape $A.B$ with quotient
  isomorphic to $C$.  Finally, the ATLAS writes $m$ to denote the
  finite cyclic group $\Z_m$.
\end{defn}

One extension of $T$ we wish to discuss is its non-oriented version.

\begin{defn}[Non-oriented $T$]\label{def:non-oriented}
  The map that sends every complex number $x+iy$ to its complex
  conjugate $x-iy$ is an orientation-reversing map of order~$2$ that
  preserves the Farey tessellation of the unit disc.  The
  \emph{non-oriented version of Thompson's group $T$} is the
  automorphism group of $T$ and it is generated by $T$ acting on
  itself by conjugation and the automorphism of $T$ induced by the
  complex conjugation map of the Farey tessellation.  It is denoted
  $T^{no}$ in \cite{Fo-assoc} and its structure is isomorphic to $T
  \rtimes \Z_2$.  In other words, it has shape $T.2$.
\end{defn}

The other groups related to $T$ that we need to discuss are the
subgroups of $T$ which centralize elements of finite order.  These
torsion elements of $T$ are easy to describe and classify.

\begin{defn}[Torsion elements]\label{def:torsion}
  Thompson's group $T$ contains torsion elements of every order and
  all finite cyclic subgroups of the same order are conjugate.  More
  explicitly, torsion elements are created by, essentially, finding an
  $m$-sided ideal polygon in the Farey tessellation whose boundary
  arcs are Farey arcs and then retriangulating this $m$-gon so that it
  looks like a $\frac{2\pi}{k}$ rotation of the original triangulation
  for some $k$ that divides $m$.  This is an element whose order
  divides $k$ and all finite cyclic subgroups of order~$k$ are
  generated by a $\frac{2\pi}{k}$ rotation of some Farey $m$-gon in
  this way.  This characterization can be used to prove that all such
  subgroups are conjugate \cite{BCST09,Fo-diss}.
\end{defn}

In our classification of automorphisms of the infinite cyclohedron, we
need to understand the centralizer of a particular order~$2$ element
in $T$.  A proof of the following result can be found in
\cite[Chapter~$7$]{Ma08}.

\begin{prop}[Centralizers]\label{prop:central}
  Let $g$ be an element in $T$ of order $m$ and let $G$ be the
  centralizer of $g$ in $T$, i.e. the elements in $T$ that commute
  with $g$.  The cyclic group $\Z_m$ generated by $g$ is a normal
  subgroup of $G$ and the quotient group $G/\Z_m$ is isomorphic to $T$
  itself.  In other words, the subgroup $G \subset T$ has shape $m.T$.
\end{prop}

The fact that $T$ has proper subgroups of shape $m.T$ for every $m$
means that it is a very self-similar group.

\section{Automorphism groups}\label{sec:auts}

In this final section we analyze the structure of the automorphism
group of the infinite polytope $\cyc^\infty$.  We first define what we
mean by an automorphism of an infinite polytope.  For ordinary
finite-dimensional convex polytopes, there are various notions of
equivalence.  Convex polytopes can be considered as metric objects, or
as objects up to equivalence under affine transformations, or as
purely combinatorial objects.  When constructing the infinite
cyclohedron inside the infinite associahedron we adopted the metric
viewpoint, but when we consider its automorphism group we adopt a
combinatorial one.

\begin{defn}[Automorphisms]\label{def:auts}
  By a \emph{combinatorial automorphism} of a convex polytope $P$, in
  finite or infinite dimensions, we mean an automorphism of its
  associated face lattice.  Thus vertices are sent to vertices, edges
  to edges, and so on, but there is no requirement that parallel faces
  be send to parallel faces as there would be if we required the
  existence of an underlying affine transformation of the ambient
  affine space.  The set of all combinatorial automorphisms clearly
  form a group $\aut(P)$, which we call its \emph{combinatorial
    automorphism group}.
\end{defn}

In this notation, the main result proved by the first author in
\cite{Fo-assoc} is the following.

\begin{thm}[Associahedral automorphisms]\label{thm:ass-aut}
  The combinatorial automorphism group of the infinite associahedron
  is isomorphic to $T \rtimes \Z_2$, the non-oriented version of
  Thompson's group $T$.  In other words, $\aut(\ass^\infty)$ has shape
  $T.2$.
\end{thm}

It might seem remarkable that the union of a nested sequence of
polytopes which only have dihedral symmetries would have such a large
and complicated automorphism group, but the situation is analogous to
the way in which the intervals $[-n,n]$ only have two symmetries each
(because the origin must be fixed) while their union, the real line,
has a much larger symmetry group which includes translations.  The
corresponding fact for Thompson's group $T$, proved by the first
author in \cite{Fo11}, is that it contains the group $PSL_2(\Z)$ as an
undistorted subgroup.  When $PSL_2(\Z)$ is viewed as the group of
orientation-preserving maps of the hyperbolic plane that preserve the
rational Farey tessellation, the connection with $T$ is clear and the
undistorted natural of the inclusion map is understandable.  

A key step in the proof of Theorem~\ref{thm:ass-aut} is the analysis
of (in our language) the structure of (the $3$-skeleton of) a
neighborhood of a vertex in $\ass^\infty$.  The first author and
Maxime Nguyen show in \cite{FoNg12} that if $\phi$ and $\psi$ are two
combinatorial automorphisms of $\ass^\infty$ and there is a vertex $v$
such that $\phi(v) = w = \psi(v)$ and the $\phi$ and $\psi$ agree on
the link of $v$, then $\phi = \psi$.  In other words, combinatorial
automorphisms are completely determined by their behavior in the
neighborhood of a single vertex.  Moreover, each feasible local
isometry is realized by an element of the non-oriented version of
Thompson's group $T$, so these are the only combinatorial
automorphisms.  What we need is an analogous fact.

\begin{prop}[Extending automorphisms]\label{prop:extend}
  Every combinatorial automorphism of the infinite cyclohedron
  $\cyc^\infty$ comes from an automorphism of the infinite
  associahedron $\ass^\infty$ that commutes with the $\pi$-rotation
  $\tau$.
\end{prop}

\begin{proof}
  The proof is similar to the one in Section~$4$ of \cite{FoNg12} but
  with minor variations.  The vertices of $\cyc^\infty$ can be
  identified with finite centrally symmetric retriangulations of the
  Farey tessellation.  The edges of $\cyc^\infty$ correspond to a flip
  either of the diagonal edge (the unique arc that passes through the
  origin of the disc) or a pair of flips of arcs symmetric with
  respect to the $\pi$-rotation.  Purely from the combinatorics of the
  flips in the neighborhood of a fixed vertex, one can recover the
  underlying structure of the tessellation.  For example, the
  $2$-cells that contain a fixed vertex $v$ are either squares,
  pentagons or hexagons.  Squares correspond to flips involving arcs
  that do not bound a common triangle.  Pentagons correspond to flips
  involving two pairs of centrally symmetric arcs where one arc in the
  first pair bounds a common triangle with one arc in the second pair.
  Hexagons correspond to the diagonal flip and one of the two flips
  involving centrally symmetric arcs bounding a common triangle with
  the diagonal arc.  In other words, the vertex $v$ belongs to the
  boundary of exactly two hexagons and these hexagons share a unique
  edge, the edge corresponding to the diagonal flip.

  Building from this identification of the diagonal flip, the
  pentagons can be used to identify arcs sharing a common triangle and
  looking at how they are included into the $3$-cells lets one see
  which endpoint they share.  Continuing in this way, one uses the
  purely combinatorial data in the low dimensional skeleton of the
  neighborhood of a vertex $v$ to identify each edge leaving $v$ with
  the flip of a specific arc (or centrally symmetric pair of arcs) in
  the Farey tesselation corresponding to $v$.  And once one has
  reconstructed the tessellation, we see that combinatorial
  automorphisms are once again completely determined by their behavior
  in the neighborhood of a single vertex.  Moreover, each feasible
  local isometry is realized by an element of the non-oriented version
  of Thompson's group $T$ which commutes with the $\pi$-rotation
  $\tau$, so these are the only combinatorial automorphisms.
\end{proof}

We are now ready to prove our main result.

\begin{thm}[Cyclohedral automorphisms]
  The combinatorial automorphism group of the infinite cyclohedron is
  isomorphic to the centralizer of a particular involution in the
  combinatorial automorphism group of the infinite associahedron
  $\aut(\ass^\infty)$.  And since the centralizer of an involution in
  $T$ is a group isomorphic to a degree~$2$ central extension of $T$,
  the group $\aut(\cyc^\infty)$ has shape $2.T.2$.
\end{thm}

\begin{proof}
  By Proposition~\ref{prop:extend} it suffices to find which elements
  in $T \rtimes \Z_2$ stabilize the polytope $\cyc^\infty$.  In order
  to send centrally symmetric triangulations to centrally symmetric
  triangulations, the corresponding map on the boundary circle must be
  invariant under the $\pi$-rotation map $\tau$.  This map $\tau$
  represents an element in $T$ of order~$2$ and its centralizer in
  $T$, by Proposition~\ref{prop:central} has shape $2.T$.  This
  represents all of the orientation-preserving elements in $T \rtimes
  \Z_2$ that stabilize $\cyc^\infty$.  Since the orientation-reversing
  complex conjugation map, the one which reflects across the $x$-axis,
  also commutes with the $\pi$-rotation map $\tau$, the full subgroup
  of $\aut(\ass^\infty)$ that stabilizes $\cyc^\infty$ is a semidirect
  product of the centralizer in $T$ and $\Z_2$, a group with shape
  $2.T.2$.
\end{proof}

\subsection*{Acknowledgements.}
This project is the result of a conversation between the authors
during the 2013 Durham symposium on geometric and cohomological group
theory and we would like to thank the organizers of that conference
for providing the institutional environment that helped produce these
results.  The second author would also like to thank Sonja Mitchell
Gallagher who recently completed a dissertation on Thompson's group
$T$ under his supervision.  Several ideas initially discussed with
Sonja are scattered throughout the article.

\newcommand{\etalchar}[1]{$^{#1}$}
\def\cprime{$'$}
\providecommand{\bysame}{\leavevmode\hbox to3em{\hrulefill}\thinspace}
\providecommand{\MR}{\relax\ifhmode\unskip\space\fi MR }
\providecommand{\MRhref}[2]{%
  \href{http://www.ams.org/mathscinet-getitem?mr=#1}{#2}
}
\providecommand{\href}[2]{#2}

\end{document}